\documentclass[12pt]{article}

\usepackage[a4paper, total={6.4in, 8.5in}]{geometry}
\usepackage{amsmath,amssymb,latexsym,amsfonts}
\usepackage{stmaryrd}
\usepackage{shapepar,bm}
\usepackage[dvips,arrow,matrix,ps,color,line]{xy}
\usepackage{theorem,amsfonts,amssymb,amscd}
\usepackage{geometry}
\usepackage{makeidx}
\usepackage{eulervm}
\usepackage{mathrsfs}
\usepackage{enumitem}

\usepackage{graphicx}
\usepackage{graphics}

\usepackage{epstopdf}

\usepackage{amssymb,graphics,hyperref}

\hypersetup{
colorlinks = true,
linkcolor = blue,
anchorcolor = red,
citecolor = blue,
filecolor = red,
pagecolor = red,
urlcolor = blue}

\newcommand{\mathsym}[1]{{}}

\DeclareMathOperator{\Spec}{Spec}

\DeclareMathOperator{\rank}{rank}

\renewcommand{\k}{\mathbf{k}} % ground field k
\renewcommand{\O}{\ensuremath{\mathcal{O}}}% structural
\renewcommand{\l}{\ensuremath{\ell}}% gap l
\renewcommand{\dim}{\ensuremath{\mathrm{dim}}}% 
\newcommand{\N}{\mathcal{S}} % semigroup notation
\renewcommand{\k}{\mathbf{k}} % ground field k
\usepackage[usenames]{color}
\definecolor{MyLightMagenta}{cmyk}{0.1,0.8,0,0.1}
\definecolor{MyDarkBlue}{rgb}{0.1,0,0.3}
\theoremstyle{change}
\theorembodyfont{\rmfamily}
\newtheorem{teo}{Theorem.\ }[section]
\newtheorem{defi}[teo]{Definition.\ }
\newtheorem{exam}[teo]{Example.\ }
\newtheorem{prop}[teo]{Proposition.\ }

\newtheorem{lema}[teo]{Lemma.\ }
\newtheorem{cor}[teo]{Corollary.\ }

\newtheorem{conj}{Conjecture.\ }[section]

\newenvironment{proof}{\paragraph{Proof.}}{\hfill$\square$}

\global
\global

\def\no@breaks#1{{\def\\{ \ignorespaces}#1}}    % disallow explicit line breaks

\makeatletter
\def\cleardoublepage{\clearpage\if@twoside \ifodd\c@page\else
\hbox{} \thispagestyle{empty}
\newpage
\if@twocolumn\hbox{}\newpage\fi\fi\fi} \makeatother

\usepackage{eso-pic,graphicx}
\makeatletter
\newcommand\BackgroundPicture[2]{%
\setlength{\unitlength}{1pt}%
default \put(0,\strip@pt\paperheight){%
\parbox[t][\paperheight]{\paperwidth}{%
\vfill
 \centering \includegraphics[angle=#2, width=15cm, height=15cm,  bb=0 0 150 150]{#1}
\vfill
}}} %
\makeatother
\makeindex

\begin{document}

\date{}

\title{On the Normal Sheaf of Gorenstein Curves}
\author{Andr\'e Contiero, Aislan Leal Fontes \& J\'unio Teles
\thanks{ The first author was partially sponsored by Funda\c c\~ao de Amparo a Pesquisa do Estado de Minas Gerais (FAPEMIG) grant no. APQ-00798-18. The third author was financed in part by the Coordena\c c\~ao de Aperfei\c coamento de Pessoal de N\'ivel Superior - Brasil (CAPES) - Finance Code 001.
\smallskip 
\newline  ${}$
\newline ${}$ \,\,\,\,\,\, {\bf 2020 MSC:} 14H20, 14H45 \& 14H51.}}

\maketitle

\begin{abstract}
We show that any tetragonal Gorenstein integral curve is a complete intersection in its respective $3$-fold rational normal scroll $S$, implying that the normal sheaf on $C$ embedded in $S$, and in $\mathbb{P}^{g-1}$ as well, is unstable for $g\geq 5$, provided that $S$ is smooth.
We also compute the degree of the normal sheaf of any singular
reduced curve in terms of the Tjurina and Deligne numbers,  providing a semicontinuity
of the degree of the normal sheaf over suitable deformations, revisiting classical results of
the local theory of analytic germs.
\noindent 

\end{abstract}

\section{Introduction}\label{Intro}

It is well known that a non-hyperelliptic Gorenstein curve $C$ of arithmetical genus $g>2$
can be embedded in the projective space $\mathbb{P}^{g-1}$ via its dualizing sheaf. Thus, $C$ becomes a canonical Gorenstein curve,
i.e. has genus $g$ and degree $2g-2$. More recently, Kleimann \& Martins in \cite{KM09}, followed by Lara, Martins \& Souza in \cite{LMS19},
show that if a singular curve has gonality $k$, then its \emph{canonical model} lies on a $(k-1)$-fold rational normal scroll of degree
$g-k+1$. In particular, assuming that $C$ is Gorenstein, its canonical model coincides with that one given by the dualizing sheaf.
Since the normal sheaf of a curve encodes many geometrical information, it is only natural to ask about the stability of the normal sheaf 
of a canonical Gorenstein curve considered in its both natural ambient spaces,
rational normal scrolls and $\mathbb{P}^{g-1}$.

The stability of the normal sheaf $\mathcal{N}_{C/\mathbb{P}^{g-1}}$ of canonical curves is well known when $C$ is a smooth 
curve of genus at most $8$. The study of $\mathcal{N}_{C/\mathbb{P}^{g-1}}$ starts with a remarkable work due to Aprodu, Farkas \& Ortega \cite{AFO16}, where
the authors show that the normal sheaf of any tetragonal 
smooth curve of genus  $7$ with maximal Clifford index is stable,  showing in particular that Mercat conjecture fails for
general curves of genus $7$. With the same techniques, the authors 
also prove that the normal sheaf of general tetragonal canonical smooth curve of genus $g\geq 6$
is unstable, in particular, the normal sheaf of a general canonical smooth curve of genus $6$ is unstable.
They also establish a conjecture.

\begin{conj}[Aprodu--Farkas--Ortega]
The normal sheaf $\mathcal{N}_{C/\mathbb{P}^{g-1}}$ of a general smooth canonical curve $C$ of genus $g\geq 7$ is stable.
\end{conj}

Later on, Bruns \cite{Bru17} uses the fact that a smooth general 
canonical curve $C$ of genus $8$ is a transversal linear section of a Grassmannian $G(2,6)$ in its 
Plucker embedding, c.f. \cite{MukIde}, to show the stability of the normal sheaf in this case, confirming 
Aprodu--Farkas--Ortega conjecture for genus $8$.

In the last decades many techniques and notions 
for singular curves have been carried out, such as gonality, canonical models, Petri's analysis
and Max Noether Theorem. In this short paper we want to propose the study of the stability of the normal sheaves on singular curves, 
starting from Gorenstein ones, and trying to avoid the normalization of these curves. One difficult lies in a right notion of what a general Gorenstein curve is, without excluding the singular ones. We do it here \emph{genus by genus}.

In section \ref{section2} of this paper, we recall the notions of linear system and gonality on singular curves. We also
introduce the notion of (un)stability of the normal sheaf of a singular curve. We also give some examples by showing that the normal sheaf of canonical Gorenstein curve in $\mathbb{P}^{g-1}$ of genus $4$ and $5$ is unstable,
c.f. Example \ref{est-proj}.

In section \ref{mainsec}, we show that any tetragonal Gorenstein curve is a complete intersection in its corresponding $3$-fold rational normal scroll $S$, c.f.
Theorem \ref{intscroll}, extending a result due Schreyer \cite[Sec. 4]{Schr86}. With the main result of Section \ref{mainsec} in hands, we are able to
show that the normal sheaves $\mathcal{N}_{C/S}$ and $\mathcal{N}_{C/\mathbb{P}^{g-1}}$ are unstable for $g\geq 5$, c.f. Theorem \ref{gmaior5}, provided that the $3$-fold rational normal scroll $S$ associated to $C$ is smooth. We also provide a sufficient condition to
the $3$-fold scroll of a Gorenstein curve of genus $6$ be smooth, see Definition \ref{def1} and Lemma \ref{scroll-suave}, that conclude our section \ref{normalsec}.

One important part of studying the normal sheaf on a curve, lies on the computation of its degree. Since our curves are singular, the computation
becomes more involved when the genus grows. The section \ref{section3} of this paper, is addressed to compute the degree of the normal sheaf on any singular
reduced projective curve by means of classical invariants of the local theory of singularities, namely Tjurina and Deligne numbers, 
c.f. Theorem \ref{degnormal}. The provided formula can be very easy to manage, see Example \ref{exemplo} and Table \ref{tabela1}. As an application of Theorem
\ref{degnormal}, we prove that the degree of the normal sheaf is an upper semicontinuous function over suitable deformations, c.f. Corollary \ref{semicont}.

\bigskip

\noindent \textbf{Acknowledgements.} We are deeply grateful to Renato Vidal Martins 
for his help and many conversations on singular curves lying on scrolls. We are also grateful to Marcelo Escudeiro Hernandes for conversations and suggesting many useful references. The authors would like to express their deep appreciation to the anonymous reviewer for her/his careful and fast reading, suggestions and several corrections. 

\section{Notation and background}\label{section2}
Let $C$ be an integral projective curve defined over an algebraically closed field $\mathbf{k}$. We recall that the
\textit{degree} of a coherent sheaf $\mathcal{F}$ on $C$ is the integer 
\begin{equation}\label{defdeg}
	\deg(\mathcal{F})=\chi(\mathcal{F})-\rank(\mathcal{F})\chi(\mathcal{O}_C)
\end{equation} where $\chi$ stands for the usual Euler characteristic. By virtue of Riemann--Roch theorem for singular curves,
we easily see that
$\mbox{deg}(\mathcal{F})=\chi(\mathcal{F})-\mbox{rank}(\mathcal{F})(1-g)$, where $g$ is the arithmetical
genus of $C$. The \textit{slope} of $\mathcal{F}$ is by definition 
$$\mu(\mathcal{F}):= \dfrac{\deg \mathcal{F}}{\rank\mathcal{F}}.$$

\noindent For the sake of self-containedness, we include the proof of the following very naive result.

\begin{lema}\label{lema-deg}
Let $C$ be an integral projective curve and $\mathcal{F}$ a coherent over $C$. The following are true:
\begin{enumerate}[label=\textbf{(\roman*)}]
\item \label{item1} $\deg(\mathcal{F})=\deg(\det(\mathcal{F}));$
\item \label{item2}If $\mathcal{F}$ is a torsion sheaf, then
$\deg(\mathcal{F})=\sum_{P\in C}\dim_k(\mathcal{F}_P);$
\item \label{item3} Given an exact sequence, $0\to \mathcal{F}_1\to \mathcal{F}_2\to \cdots \to \mathcal{F}_n\to0$,
of coherent sheaves over $C$,
we have $\sum_{i=1}^{n}(-1)^i\deg(\mathcal{F}_i)=0.$
\end{enumerate}
\end{lema}
\begin{proof}
Item \ref{item1} follows automatically from equation \eqref{defdeg}. Since the rank of $\mathcal{F}$ is zero, 
and $\mathcal{F}$ is a skyscraper sheaf, we get $H^1(C,\mathcal{F})=0$ and so
$\deg(\mathcal{F})=h^0(C,\mathcal{F})$, getting \ref{item2}. The proof of \ref{item3} follows by the
additive properties of Euler's characteristic and also of the rank on exact sequences.
\end{proof}

\medskip

We also recall that the notion of linear systems on singular curves
is characterized by interchanging line bundles by torsion free sheaves of rank $1$.
A \emph{linear system of dimension $r$ and degree $d$} on a curve $C$, 
possibly singular, is a set of the form
$$\mathfrak{g}^{r}_{d}:=\{x^{-1}\mathcal{F}\ |\ x\in V\setminus 0\}$$
where $\mathcal{F}$ is a fractional ideal sheaf of degree $d$ on $C$ and $V$ is a vector subspace of 
$H^{0}(C, \mathcal{F} )$ of dimension $r+1$. Note that \emph{non-removable} base points are allowed.

The \emph{gonality of $C$} is the smallest $k$ for which there exists a 
$\mathfrak{g}_{k}^{1}$ on C, or equivalently, a torsion free sheaf $\mathcal{F}$ of rank $1$ on $C$ with degree $k$ and 
$\dim \mathrm{H}^{0}(\mathcal{C, F})\geq 2$. We also note that singular curves may admit linear systems of degree
bigger than $\left\lfloor (g+3)/2\right\rfloor$, c.f. \cite{LMS19}.

The next definition is just a simple adjustment of the notion of stability, and semi-stability, of a coherent sheaf on a smooth curve
in a way to also include singular curves.

\begin{defi}
A coherent sheaf $\mathcal{F}$ on an integral projective curve, possibly singular, is
\textit{semistable} (resp. \textit{stable}) if for each nontrivial coherent subsheaf  
$0\neq \mathcal{G}\subsetneq \mathcal{F}$, with $\rank\mathcal{G}<\rank\mathcal{F}$,  one has 
$\mu(\mathcal{G})\leq \mu(\mathcal{F})$ (resp. $\mu(\mathcal{G})<\mu(\mathcal{F})$). Moreover, 
$\mathcal{F}$ is said to be \textit{polistable} if it can be written as a sum of stable subsheaves, all of the same slope.
As usual, a sheaf is \emph{unstable} if it is not stable.
\end{defi}

\begin{exam}\label{exemplane}
If $C$ is an integral plane curve of degree $d$, then $\mathcal{N}_{C/\mathbb{P}^{2}}\cong \mathcal{O}_{C}(d)$.
Once $\mathcal{N}_{C/\mathbb{P}^{2}}$ is torsion free, and every coherent sheaf of rank zero is a torsion sheaf, 
we get that  $\mathcal{N}_{C/\mathbb{P}^{2}}$ is stable. In particular, the normal sheaf of a nonhyperelliptic
canonical Gorenstein curve of genus $3$ is stable.
\end{exam}

\medskip

The moduli space $\mathcal{M}_g$ of smooth curves of genus $g\geq 2$ admits a filtration 
$$\mathcal{H}_g:=\mathcal{M}_{g}(2)\subset \mathcal{M}_{g}(3)\subset \dots\subset \mathcal{M}_g\left(\left\lfloor (g+3)/2\right\rfloor\right)=\mathcal{M}_g,$$
where $\mathcal{M}_{g}(k):=\{[C]\in \mathcal{M}_g: C \hspace{0.2cm}\mbox{admits a }\hspace{0.2cm}\mathfrak{g}_{k}^{1}\}$ 
is an irreducible closed subset of $\mathcal{M}_g$, c.f. \cite{F69}, of dimension $2g+2k-5$, see \cite[eq. 2.3, pg. 346]{AC81}. Here $\mathcal{H}_g$ stands for the space of hyperelliptic curves.
Hence, the locus of $k$-gonal smooth curves is just $\mathcal{M}^k_{g}:=\mathcal{M}_{g}(k)\setminus\mathcal{M}_g(k-1)$.
If $k\geq \left\lfloor(g+3)/2\right\rfloor$, then $\mathcal{M}_{g}(k)=\mathcal{M}_g$,
see \cite{ACGH85}.  In this way, for smooth curves of genus $g$, the number $\left\lfloor(g+3)/2\right\rfloor$ is called \textit{generic gonality}, and usually, a \emph{general smooth curve} is defined in terms of the generic gonality.

In this paper we also define a general possibly singular curve with a fixed arithmetic genus $g$ in terms of the number $\left\lfloor(g+3)/2\right\rfloor$.
But it is also required, as the genus increases, to avoid also some specific cases, see for example the Definition \ref{def1} below. 

\begin{defi}
A general curve of arithmetic genus $4$ or $5$ is just an integral Gorenstein curve whose gonality is $3$ or $4$, respectively. 
\end{defi}

\begin{exam}\label{est-proj}
If $C$ is a general curve of genus $4$, then Petri's Theorem for Gorenstein curves, c.f. \cite[Section 3]{Sto93-2},
assures that $C$ is a complete intersection 
of a quadric with a cubic hypersurface in $\mathbb{P}^3$. Thus, the normal sheaf splits
$$\mathcal{N}_{C/\mathbb{P}^{3}}\cong \mathcal{O}_{C}(2)\oplus \mathcal{O}_{C}(3).$$
%omiti a referência an c.f. \cite[Cor. 1.3.]{Schr91}
The nontrivial subsheaf $\mathcal{G}:=\mathcal{O}_{C}(3)$ of $\mathcal{N}_{C/\mathbb{P}^{3}}$ 
has slope $\mu(\mathcal{G})=18$ while
$$\det(\mathcal{N}_{C/\mathbb{P}^{3}})\cong \mathcal{O}_{C}(2)\otimes \mathcal{O}_{C}(3)\cong \mathcal{O}_{C}(5),$$
and so $\mu(\mathcal{N}_{C/\mathbb{P}^{3}})=30/2=15$. Hence,  $\mathcal{N}_{C/\mathbb{P}^{3}}$ is unstable. 
Now, we consider $C$ a general curve of genus 
$5$. Again, by Petri's Theorem for Gorenstein curves, $C$ is a complete 
intersection of three quadrics and so its normal sheaf can be written as sum
$$\mathcal{N}_{C/\mathbb{P}^{4}}\cong \mathcal{O}_{C}(2)\oplus \mathcal{O}_{C}(2)\oplus \mathcal{O}_{C}(2).$$
By taking the subsheaf $\mathcal{G}:=\mathcal{O}_{C}(2)\subsetneq\mathcal{N}_{C/\mathbb{P}^{4}}$ 
we get $\mu(\mathcal{G})=16=48/3=\mu(\mathcal{N}_{C/\mathbb{P}^{4}})$, hence the 
normal sheaf $\mathcal{N}_{C/\mathbb{P}^{4}}$ is also unstable. 

So we may conclude that
the normal sheaf $\mathcal{N}_{C/\mathbb{P}^{g-1}}$ of a general canonical curve of arithmetic genus $4$ or $5$ is unstable. In addition, the normal sheaf of a general canonical curve of genus $5$ is polistable.
\end{exam}

For further use, we finish this section presenting a naive generalization of
a result present in Hartshorne's book \cite[Thm. 7.11]{Har77}. 

Let $X\subseteq\mathbb{P}^{n}$ be a regular scheme of codimension $q$, and $Y\subseteq X$ be a closed 
subscheme of codimension $p$ in $X$. The dualizing sheaf 
$\omega_Y$ is given by $\omega_Y:=\mathcal{E}xt_{\mathcal{O}_{\mathbb{P}^n}}^{q}(\mathcal{O}_Y,\omega_{\mathbb{P}^{n}}).$
A result due to Grothendieck, c.f. \cite[Prop. 5]{Gro59}, assures that for any coherent sheaf $\mathcal{F}$ on $Y$, the sheaf
$\mathcal{E}^i(\mathcal{F}):= \mathcal{E}xt_{\mathcal{O}_{\mathbb{P}^{n}}}^{q+i}(\mathcal{F},\omega_{\mathbb{P}^{n}})$
does not depend on the regular space in which $Y$ is considered. So we may interchange $\mathbb{P}^n$ wit
the regular projective scheme. 
The conclusion is that the dualizing sheaf of $Y$ is also given by
$$\omega_Y= \mathcal{E}xt_{\mathcal{O}_X}^{p}(\mathcal{O}_Y,\omega_{X}).$$
Following the steps in an entirely similar way of
\cite[Thm. 7.11]{Har77}, just replacing the projective space $\mathbb{P}^{n}$ by a regular variety $X$, one can establish the next result.

\begin{lema}\label{dual-lci}
Let $Y$ be a closed subscheme of a $n$-fold regular projective scheme $X$. Suppose that $Y$ is a 
locally complete intersection of codimension $p$ in $X$, we have
$$\omega_Y=\omega_X\otimes \mathcal{O}_Y\otimes \det(\mathcal{N}_{Y/X}).$$
\end{lema}

\section{Tetragonal Gorenstein curves}\label{mainsec}

A nonhyperelliptic Gorenstein curve $C$ of genus $g>2$ has two natural ambient spaces to embed it, namely the projective
spaces $\mathbb{P}^{g-1}$ and suitable \emph{rational normal scrolls}. While the projective space depends only on its
genus $g$, it is also required to know the gonality of $C$ to embedding it in a suitable scroll. 
So we briefly recall some basic facts on scrolls.

Given $d>1$ non-negative integers $e_1\leq \dots\leq e_d$, set $e=\sum e_i$ and $N=e+d-1$. The 
rational normal scroll $S:=S(e_1,\dots,e_d)\subseteq\mathbb{P}^{N}$ is the projective variety that, 
after a choice of coordinates, is the set of points $(x_0:\dots: x_N)\in\mathbb{P}^N$ such that 
\begin{equation*}
\rank\left(
\begin{array}{ccc}
x_0 & \dots & x_{e_1-1} \\
x_1 & \dots & x_{e_1}
\end{array}\right.
\left.
\begin{array}{ccc}
x_{e_1+1} & \dots & x_{e_1+e_2} \\
x_{e_1+2} & \dots & x_{e_1+e_2+1}
\end{array}\right.
\left.
\begin{array}{c}
\dots  \\
\dots
\end{array}\right.
\left.
\begin{array}{ccc}
x_{N-e_d} &\dots & x_{N-1}  \\
x_{N-e_d+1} & \dots & x_{N}
\end{array}\right)\leq 1.
\end{equation*} By taking $d$ rational normal curves of degrees $e_1, \dots,e_d$ lying on $d$ complementary
linear spaces in $\mathbb{P}^N$, one can see that $S$ is the disjoint union  of all $(d-1)$-plane in $\mathbb{P}^N$
determined by choosing a point in each one of the $d$ rational normal curves. Hence, the dimension of $S$ is $d$ and its
degree is $e$. Each $(d-1)$-plane is called a ruling
of $S$. Additionally, $S$ is a smooth variety if, and only if,  $e_i>0$ for all $1\leq i\leq d$. 

On the other hand, taking the smooth variety 
$\mathbb{P}(\mathcal{E})$ where $\mathcal{E}=\O_{\mathbb{P}^1}(e_1)\oplus\dots\oplus\O_{\mathbb{P}^1}(e_d)$,
one can see that there is a birational morphism $\pi:\mathbb{P}(\mathcal{E})\rightarrow S$ induced by $\O_{\mathbb{P}(\mathcal{E})}(1)$,
and $\pi$ is a rational resolution of singularities of $S$. When $S$ is singular, it is a cone whose singular locus is a vertex $V$
of dimension $\#\{i\,;\,e_i=0\}-1$.
The birational morphism $\pi$ is such that each fiber of $\vartheta:\mathbb{P}(\mathcal{E})\rightarrow\mathbb{P}^1$ is mapped to a ruling.
Moreover, one can show that $\mathrm{Pic}(\mathbb{P}(\mathcal{E}))=\mathbb{Z}\tilde H\oplus\mathbb{Z}\tilde R$, where $[\tilde H]=[\O_{\mathbb{P}(\mathcal{E})}(1)]$ 
is a hyperplane section while $[\tilde R]:=[\vartheta^{*}(\O_{\mathbb{P}^1}(1))]$ is
a fiber class, satisfying
\begin{equation*}
{\tilde H}^d = e, \ \  {\tilde H}^{d-1}\tilde R=1,\ \ {\tilde R}^2 = 0 \ \mbox{and } K_{\mathbb{P}(\mathcal{E})}=(e-2)\tilde R-d{\tilde H}
\end{equation*} where $K_{\mathbb{P}(\mathcal{E})}$ is the canonical class of $\mathbb{P}(\mathcal{E})$. We also
fix the following notation
$$\O_{S}(aH+bR):=\pi_{*}\O_{\mathbb{P}(\mathcal{E})}(a\tilde H+b\tilde R),$$ where $[H]$ is a hyperplane class of $S$ and $[R]$ is a class of a ruling.
Note if $S$ is singular, then $\O_{S}(aH+bR)$ is a fractional ideal sheaf (or a divisorial sheaf) associated to a suitable Weil divisor of $S$, more precisely:
\begin{itemize}
\item if $\mathrm{codim}(V,S)>2$ then $\mathrm{Pic}(\mathbb{P}(\mathcal{E}))$ is isomorphic to the group $\mathrm{Cl}(S)$ of Weil divisors of $S$,
hence $\mathrm{Cl}(S)=\mathbb{Z}[H]\oplus\mathbb{Z}[R]$;
\item if $\mathrm{codim}(V,S)=2$ and $E$ stands the exceptional divisor of $\pi$, then sequence 
$$0\rightarrow\mathbb{Z}\xrightarrow{\cdot E} \mathrm{Pic}(\mathbb{P}(\mathcal{E}))\xrightarrow{\pi_{*}}\mathrm{Cl}(S)\rightarrow 0,$$
is exact, and in this case $E\sim \tilde H-e\tilde R$ and $\mathrm{Cl}(S)=\mathbb{Z}[R]$,
\end{itemize} c.f. \cite[Proposition 2.1 and Corollary 2.2]{Rita0}. We also can conclude that the dualizing sheaf $\omega_{S}$
is $\pi_{*}K_{\mathbb{P}(\mathcal{E})}=\O_{S}((e-2)R-dH)$.

In \cite{LMS19} the authors show that the \emph{canonical model} of any $k$-gonal singular curve 
$C$ can be embedded in a $(k-1)$-fold scroll $S$. In addition, they also show 
that the possible base points of a $\mathfrak{g}^1_k$ are exactly those lying on the vertex of $S$.
The \emph{canonical model} is defined in \cite{KM09}
and it coincides with the canonical curve in $\mathbb{P}^{g-1}$  when $C$ is Gorenstein.
Thus, it is only natural
to ask about the stability of the normal sheaves of canonical curves in $(k-1)$-fold scrolls.

It is very easy to see that a canonical smooth (or even a Gorenstein) curve is not a complete intersection in $\mathbb{P}^{g-1}$, 
provided that its genus is bigger than $5$. On the other hand, it is a hard problem to realize canonical (smooth) curves as complete intersection in known projective varieties. For instance, Mukai \cite{Muk95, Muk10} and Mukai \& Ide \cite{MukIde} realize the 
canonical model of smooth curves of small genus ($\leq 9$) as complete intersections in suitable known projective varieties, 
such as product of projective spaces, weighted projective spaces and Grassmannians.  Moreover, in \cite{Schr86} Schreyer 
proves that any tetragonal smooth curve is a complete intersection in its corresponding $3$-fold smooth Scroll $\mathbb{P}(\mathcal{E})$.  The next theorem extend the above cited result
due to Schreyer by allowing Gorenstein curves too.

\begin{teo}\label{intscroll}
A tetragonal Gorenstein curve is a complete intersection of
$Y_1\sim 2H-b_1R$ and $Y_2\sim 2H-b_2R$, with $b_1+b_2=e-2$, in the corresponding $3$-fold scroll $S=S(e_1,e_2,e_3)$.
\end{teo}
\begin{proof}
Let $C\subset\mathbb{P}^{g-1}$ be a tetragonal canonical Gorenstein curve and $S$
be the $3$-fold scroll where $C$ lies on it. Let us also take the $\mathbb{P}^2$-bundle $\mathbb{P}(\mathcal{E})$  
over $\mathbb{P}^1$, that is a resolution of singularities of $S$ by the birational morphism $\pi:\mathbb{P}(\mathcal{E})\rightarrow S$ with
exceptional divisor $F$.

Considering the pushforward $\pi_*$ , it is known that $\pi_{*}\O_{\mathbb{P}(\mathcal{E})}=\O_S$. Then we may consider
$\pi_{*}$ as a functor from the category of $\O_{\mathbb{P}(\mathcal{E})}$-modules to the
category of $\O_{S}$-modules. Since the scroll $S$ has only rational singularities, for every $j>0$ we get
\begin{equation}\label{derived}
R^{j}\pi_{*}\O_{\mathbb{P}(\mathcal{E})}(a\tilde{H}+b\tilde{R})=0, \ \mbox{whenever } a,b\in\mathbb{Z}\ \mbox{and}\ b\geq -1,
\end{equation}
where $R^{j}\pi_*$ denote the right derived functors, c.f. \cite{Vie77} and \cite[(3.5)]{Schr86}.

Let us first suppose that the $\mathfrak{g}^1_{4}$ on $C$ has no base points, then we assume that $C$ lies on $\mathbb{P}(\mathcal{E})$ just 
because
$C$ does not pass through the possible singular locus of $S$. In this way, $C$ is isomorphic to its lifting $C^{\prime}$ to
$\mathbb{P}(\mathcal{E})$. By abuse of notation, we use $C$ to denote the lifting of $C$ as well. Since $\mathbb{P}(\mathcal{E})$ is smooth, the result follows just like in \cite{Schr86}, in the 
following way. The pencil $\mathfrak{g}^1_4$ has no base points. So, if $D$ is a canonical divisor given by the intersection
of $C$ with a ruling, then the linear spam of any subscheme of $D$ with degree $3$ is not a straight line, otherwise we get a 
$\mathfrak{g}^1_3$. Hence, we are in the range of \cite[Lemma, item 3, pg.~119]{Schr86}. The geometric version of 
Riemann--Roch Theorem for Gorenstein curves assures that $C\subset\mathbb{P}(\mathcal{E})$ has constant betti numbers over 
$\mathbb{P}^1$, c.f. \cite[Prop. 4.3]{Schr86}. Finally, like in  \cite[Cor 4.4]{Schr86}, 
$C\subset\mathbb{P}(\mathcal{E})$ admits the following free resolution
\begin{equation}\label{freeresolution0}
0\rightarrow\O_{\mathbb{P}(\mathcal{E})}((e-2) \tilde R{-}4 \tilde H)\rightarrow
\O_{\mathbb{P}(\mathcal{E})}(b_1\tilde R{-}2 \tilde H)\oplus O_{\mathbb{P}(\mathcal{E})}(b_2 \tilde R{-}2 \tilde H)\rightarrow\O_{\mathbb{P}(\mathcal{E})}\rightarrow\O_{C}\rightarrow 0,
\end{equation}
where $b_1+b_2=e-2$ with $b_1,b_2\geq -1$. Taking the pushforward
$\pi_*$ on the above exact sequence \eqref{freeresolution0}, it follows from equation \eqref{derived} that we 
get a free resolution of $C$ on $S$, implying that $C$ is a complete intersection in $S$ determined by the two mentioned divisors in the statement.

Now let us assume that the pencil $\mathfrak{g}^1_4$ has base points, thus $C$ meets the vertex of $S$. 
We also may assume that the base points of $\mathrm{g}^1_{4}$ are singular points of $C$, otherwise
each such base point is removable and so $C$ is no longer tetragonal. The vertex $V$ is either a straight line or a point,
i.e. $\mathrm{codim}(V,S)\geq 2$. 

We first show that $C\subset S$ has constant betti-numbers over the rulings. 
So, let $D$ be a subcanonical divisor of degree $4$ that is the intersection of $C$ with a ruling. By considering $D$ as subscheme of dimension $0$,
by the construction of $S$ we get that the linear space spanned by $D$ in $\mathbb{P}^{g-1}$, say $\overline{D}$, is isomorphic to a $\mathbb{P}^2$.
Let $E\subset D$ be a zero dimensional subscheme of degree $3$.
If $E$ is contained in a straight line of $\mathbb{P}^2\cong\overline{D}$, then the geometric version of
the Riemann-Roch Theorem for Gorenstein curves implies that $E$ compounds a $\mathfrak{g}^1_3$, that is a contradiction
because $C$ is tetragonal. Hence, $E$ is not contained in a hyperplane of $\mathbb{P}^2$. 
Now, by virtue of \cite[Lemma 4.2]{Schr86}, 
the betti-numbers of $C$ are constants over the rulings, depending only on the gonality of $C$. % that in our particular case they all are equal to $1$.

Next we consider the lifting of $C$ to $\mathbb{P}(\mathcal{E})$ by the birational morphism $\pi$. Since the degree of the divisors
given by the intersection of the lifting with the fibers of $\mathbb{P}(\mathcal{E})$ may decrease, i.e. could be smaller than $4$, 
we then take a reducible curve $C^{\prime}$ that is the lifting of $C$ union with suitable exceptional
divisors $\mathbb{P}^1\cong E_i\subset E$ such that the intersection of $C^{\prime}$ with the fibers
is still $4$, see Figure \ref{cprime}.

\begin{figure}
\caption{Construction of $C^{\prime}$}
\centering
\label{cprime}
\includegraphics{figfinal.eps}
\end{figure}

By construction, the fibers of $\mathbb{P}(\mathcal{E})$ meet $C^{\prime}$ in
four points, satisfying the conditions of \cite[Lemma 4.2 (3)]{Schr86}, whose proof does not require a canonical curve, just a linear system over a curve in $\mathbb{P}(\mathcal{E})$. So $C^{\prime}$ has constant betti-numbers over the fibers that depend only on the number of
intersections of $C^{\prime}$ with the fibers, hence they are equal to the betti-numbers of $C\subset S$ over the rulings.
Furthermore, it is worth mention that we are also able to apply Schreyer's results \cite[Theorem 3.2]{Schr86} and  \cite[Corollary 4.4, item (i)]{Schr86} to the curve $C^{\prime}$. While the mentioned Theorem 3.2 just requires that $C^{\prime}$ has constant betti numbers over $\mathbb{P}^1$, the proof of \cite[Corollary 4.4, item (i)]{Schr86} follows directly
from \cite[Theorem 3.2]{Schr86}, depending only on the fact that $C^{\prime}$ is a curve with constant betti numbers over $\mathbb{P}^1$. Hence, we may conclude that
$\O_{C^{\prime}}$ admits a free resolution of the form
\begin{equation*}\label{freeresol1}
0\rightarrow\O_{\mathbb{P}(\mathcal{E})}((e-2)\tilde R{-}4\tilde H)\rightarrow
\O_{\mathbb{P}(\mathcal{E})}(b_1\tilde R{-}2\tilde H)\oplus O_{\mathbb{P}(\mathcal{E})}(b_2\tilde R{-}2\tilde H)\rightarrow\O_{\mathbb{P}(\mathcal{E})}\rightarrow\O_{C^{\prime}}\rightarrow 0,
\end{equation*}
where $b_1+b_2=e-2$ with $b_1,b_2\geq -1$. Taking the pushforward to $S$, and using again equation \eqref{derived}, we obtain that
\begin{equation*}
0\rightarrow\O_{S}((e-2)R-4H)\rightarrow\O_{S}(b_1R-2H)\oplus O_{S}(b_2R-2H)
\rightarrow\O_{S}\rightarrow\O_{C}\rightarrow 0
\end{equation*} 
is a free resolution of $\O_{C}$ and we are done.
\end{proof}
%%%%%

\section{The normal sheaf of a tetragonal curve}\label{normalsec}

If a Gorenstein curve $C$ has gonality $3$, then it is determined by a single divisor on the associated $2$-fold scroll, 
automatically implying that $\mathcal{N}_{C/S}$ is stable. The next step is then to consider the normal sheaf of tetragonal Gorenstein curves.

\begin{teo}\label{gmaior5}
Let $C$ be a tetragonal Gorenstein curve of genus $g\geq 5$ and $S$ its associated $3$-fold scroll. If $S$ is smooth, then $\mathcal{N}_{C/S}$ and $\mathcal{N}_{C/\mathbb{P}^{g-1}}$ are unstable.
\end{teo}
\begin{proof}
By virtue of Theorem \ref{intscroll}, $C$ is a complete intersection in $S$, say that $C$ is determined by the divisors $Y_1$ and $Y_2$. So the normal
sheaf $\mathcal{N}_{C/S}$ splits as a sum of two subsheaves, namely $\mathcal{N}_{C/S}=\mathcal{N}_{Y_1/S}\oplus\mathcal{N}_{Y_2/S}$. Since $\mathrm{rank}(\mathcal{N}_{Y_i/S})=1$, one of these two rank one subsheaves destabilize $\mathcal{N}_{C/S}$.

Now we move our attention to the normal sheaf $\mathcal{N}_{C/\mathbb{P}^{g-1}}$. In this case $C$ is no longer a 
complete intersection whenever $g\geq 6$. The way we choose to show that $\mathcal{N}_{C/\mathbb{P}^{g-1}}$ is also unstable,
requires to compute the degree of $\mathcal{N}_{C/S}$.

Let $H$ and $R$ be the hyperplane and the ruling sections of $S$. Since $C$ is a tetragonal canonical curve, $R\cdot C=4$ and $\deg(\omega_{C})=2g-2=H\cdot C$. Thus, by Lemma \ref{dual-lci} we get
\begin{eqnarray*}
\deg(\mathcal{N}_{C/S})=\deg(\omega_{C})-\deg(K_{S}\otimes\O_C)=
4H\cdot C-(g-5)R\cdot C=4g+12,
\end{eqnarray*} and so $\mu(\mathcal{N}_{C/S})=2g+6$. 
Since $C$ is a complete intersection in $S$ and $S$ is smooth, the following sequence is exact,
$$0\longrightarrow \mathcal{N}_{C/S}\longrightarrow \mathcal{N}_{C/{\mathbb{P}^{g-1}}}\longrightarrow
	\mathcal{N}_{S/{\mathbb{P}^{g-1}}}\otimes \mathcal{O}_C\longrightarrow0.$$ Taking degrees we get $\deg(\mathcal{N}_{C/\mathbb{P}^{g-1}})=\deg(\mathcal{N}_{C/S})+\deg(\mathcal{N}_{S/\mathbb{P}^{g-1}}\otimes\mathcal{O}_C)$. On the other hand, by the 
flatness of $\mathcal{N}_{S/{\mathbb{P}^{g-1}}}$, we get the exact
sequence $$0\longrightarrow \mathcal{T}_S\otimes \mathcal{O}_C\longrightarrow \mathcal{T}_{\mathbb{P}^{g-1}}\otimes \mathcal{O}_C\longrightarrow
\mathcal{N}_{S/{\mathbb{P}^{g-1}}}\otimes \mathcal{O}_C\longrightarrow0.$$
So we may conclude that
\begin{eqnarray*}
\deg(\mathcal{N}_{S/{\mathbb{P}^{g-1}}}\otimes \mathcal{O}_C)=\deg(\mathcal{T}_{\mathbb{P}^{g-1}}\otimes \mathcal{O}_C)-\deg(\mathcal{T}_S\otimes \mathcal{O}_C)= \\
-\deg(K_{\mathbb{P}^{g-1}}\otimes \mathcal{O}_C)+\deg(K_S\otimes \mathcal{O}_C)=g(2g-2)-2g-14=2(g^2-2g-7).
\end{eqnarray*}
Hence, $\deg(\mathcal{N}_{C/\mathbb{P}^{g-1}})=2(g-1)(g+1)$ and $\mu(\mathcal{N}_{C/\mathbb{P}^{g-1}})=2(g-1)(g+1)/(g-2)$, and then
$$2(g-1)(g+1)/(g-2)=\mu(\mathcal{N}_{C/\mathbb{P}^{g-1}})\leq\mu(\mathcal{N}_{C/S})=2(g+3)\ \mbox{for}\   \,g\geq 5,$$ finishing the proof.
\end{proof}
	
\medskip

We now move our attention to provide sufficient conditions to assure that the
$3$-fold scroll of a tetragonal Gorenstein curve of genus $6$ is smooth. It is known that the ideal of any nonhyperelliptic canonical Gorenstein curve is 
generated by quadratic forms, with the exception of a quintic plane curve of genus $6$, c.f. \cite[Section 3]{CF18}. 
Hence, the divisors $Y_i\in\mathrm{Div}(S)$ in the Theorem \ref{intscroll} are also given by quadratic forms on $S$, say
\begin{equation}
\mathfrak{Q}_i=\sum_{1\leq j\leq k\leq 3}P_{ijk}\sigma_j\sigma_k, \ \ (\mbox{with} \ i=1,2),
\end{equation}where we pickup suitable sections $\sigma_i \in H^0(S, \mathcal{O}_S(H-e_jR))$ corresponding to the
rational normal curves, and each $P_{ijk}$
is a homogeneous polynomial in $\k[s,t]$ of degree $e_j+e_k-b_i$.
\begin{lema}\label{biel-quintica}
Let $C$ be a tetragonal Gorenstein curve canonically embedded in $\mathbb{P}^{g-1}$.
If $S:=S(e_1,e_2,e_3)$ is the $3$-fold scroll where $C$ is a complete intersection of the quadratic forms $Y_1$ and $Y_2$
as above, then 
\begin{enumerate}[label=\textbf{(\alph*)}]
\item \label{itema1}$b_1\leq 2e_2$ and $b_2\leq 2e_3$;
\item \label{itema2}$e_3=b_2=0$ if, and only if,  $C$ is bi-elliptic;
\item \label{itema3}If $e_3=0$ and $b_2=-1$, then $C$ has a $\mathfrak{g}^2_5$.
\end{enumerate}
\end{lema}
\begin{proof}

Item \ref{itema1} was proved by Brawner in \cite[Proposition 3.1]{Bra97} just applying
Riemann--Roch theorem and the irreducibility of $C$, so it works if $C$ is singular as well. Item \ref{itema2}
was also provide by Brawner, \cite[Proposition 3.2]{Bra97}, provided that $C$ is smooth. Here we just have to
make minor adjusts to reach the Gorenstein case.

Assuming that $e_3=0$,  $S$ is a cone over a scroll of dimension $2$, and
from $b_2=0$,  the intersection of $Y_1$ with a ruling is a conic.
Hence, the linear system $\mathfrak{g}^1_4$ is given by composition of two double covers
$C\rightarrow E\rightarrow \mathbb{P}^1$ and $Y_1$ is a birational ruled surface over 
$E$ with a rational curve $\tilde{E}$ of double points. Thus, intersecting $\tilde{E}$ with $Y_2$ we have that
$$2m=\tilde{E}\cdot (2H-b_2R)=2\deg(\tilde{E}),$$
and $m=0$ if, and only if, $C$ is smooth. The intersection of $H$ with $Y_1$ is a curve birational to $E$, and its arithmetic genus is $g(E)=1$, %given by $p_aE=\frac{b_2}{2}+1$ 
because
$$2g(E)-2=H(2H-b_1R)(f-2-b_1)R=0.$$
Hence, $C$ is bi-elliptic. The proof of the converse of item \ref{itema2} follows exactly as in \cite[Proposition 3.2]{Bra97}. It remains to prove item \ref{itema3}. Since $b_1=e-1$, $C$ admits a $\mathfrak{g}^1_3$ or a $\mathfrak{g}^2_5$. To see this, we just have to adapt the arguments in \cite[pg. 128]{Schr86} using that each $Y_i$ is given by a quadratic form in $S$, the intersection theory of $S$ and Riemann--Roch theorem for singular curves. By the hypothesis, $C$ does not admit a $\mathfrak{g}^1_3$, hence $C$ has a $\mathfrak{g}^2_5$. 
\end{proof}

\begin{defi}\label{def1}
A general curve of arithmetic genus $6$ is a non-bielliptic tetragonal integral Gorenstein curve not admitting a $\mathfrak{g}^2_5$.
\end{defi}

\begin{lema}\label{scroll-suave}
The corresponding scroll of a general curve $C$ of arithmetic genus $6$ is smooth.
\end{lema}
\begin{proof}
The scroll $S$ can be of three types, namely $S(3,0,0)$, $S(2,1,0)$ or $S(1,1,1)$. 
If $C$ is in a scroll with $e_3=0$ (resp. $e_2=0$) then by item \ref{itema1} of Lemma \ref{biel-quintica} 
we get $b_2\leq0$. Hence, 
$b_2=0$ or $b_2=-1$. If $b_2=0$ (resp. $b_2=-1$), then by 
\ref{itema2} (resp. \ref{itema3}) of the Lemma \ref{biel-quintica} follows that the curve $C$ 
is bi-elliptic (resp. $C$ has a $\mathfrak{g}^2_5$). In both cases $C$
is not a general curve as in Definition \ref{def1}. Therefore,  $S=S(1,1,1)$. 
%For$b_2=-1$ the item \ref{itema3} implies that the curve $C$ has a $\mathfrak{g}^2_5$
\end{proof}

\medskip

We finish this section pointing out two naive declinations of the above results.

%%%%%%
\begin{cor}\label{est-scroll} 
The normal sheaf $\mathcal{N}_{C/S}$ of a general curve $C$ embedded in the associated scroll $S$ is 
stable for arithmetic genera $3$ and $4$, and unstable for $5$ and $6$.
\end{cor}

\begin{cor}\label{genero6}
The normal sheaf $\mathcal{N}_{C/{\mathbb{P}^{5}}}$ of a general curve $C$ of arithmetic genus $6$ is unstable.
\end{cor}

\section{Local-global method}\label{section3}

An important part of studying the stability of the normal sheaf lies in the computation of its degree. Since
here we assume Gorenstein singularities, this computation tends to become more involved when the genus grows. So
local methods can be useful, addressing to the study of the cotangent complex and suitable local invariants, like Deligne and Tjurina numbers. 
In the very beginning of this section we just fix some required notation for remaining of the paper.

Let $C\subset\mathbb{P}^r$ be a reduced projective curve defined over an algebraically closed field $\textbf{k}$ of
characteristic zero. When  
$C$ is not smooth, the fundamental sequence involving the tangent and the normal sheaves 
%$0\rightarrow \mathcal{T}_{C}\rightarrow \mathcal{T}_{C/\mathbb{P}^r}\rightarrow \mathcal{N}_{C/\mathbb{P}^r}$
is no longer a short exact sequence, but it
extends to an exact sequence with four terms
\begin{equation}\label{seq-us}
0\rightarrow \mathcal{T}_{C}\rightarrow \mathcal{T}_{C/\mathbb{P}^r}\rightarrow \mathcal{N}_{C/\mathbb{P}^r}\rightarrow \mathrm{T}^1\rightarrow 0\,,
\end{equation}
where $\mathrm{T}^1:=\mathrm{T}^1_{C/\k}$ stands for the first cohomology module of the cotangent complex of $C$, see
Lichtenbaum--Schlessinger remarkable paper \cite{LS67} for the precise definitions. Furthermore, 
\cite[Lemma 3.1.2]{LS67} assures that the $\mathcal{O}_C$-module 
$\mathrm{T}^1$ can be taken as the cokernel of the map 
$\mathcal{T}_{C/\mathbb{P}^r}\rightarrow \mathcal{N}_{C/\mathbb{P}^r}$.
%Since the space of the global sections of the three sheaves $\mathcal{T}_{C}\,, \mathcal{T}_{C/\mathbb{P}^r}$ 
%and $\mathcal{N}_{C/\mathbb{P}^r}$ have finite dimensional, 

Given a singular point $P\in C$, $\tau_P:=\dim_k(\mathrm{T}^1_{P})$ stands for the 
\emph{Tjurina number} at $P\in C$, see \cite{Greuel} for the connection between the Tjurina algebra and the
cotangent complex of an analytic germ. As an immediate consequence of Lemma \ref{lema-deg}, we get
$$\deg(\mathrm{T}^1)=\sum_{P\in \mathrm{Sing}(C)}\tau_{P}=:\tau,$$
that is the so-called  \textit{Tjurina number of $C$}. We also recall that the 
\textit{singularity degree of $C$ at $P$} is defined by $\delta_P:=\dim_{\k} \overline{\mathcal{O}}_{C,P}/\mathcal{O}_{C,P}$, where $\overline{\mathcal{O}}_{C,P}$ is the normalization of $\O_{C,P}$, and then we set
$\delta:=\sum_{P\in C}\delta_P$, the singularity degree of $C$.

We also consider the skyscraper sheaf $\mathfrak{D}_C$ such that its stalk at any $P\in C$ is
\begin{equation}\label{Ddefi}
\mathfrak{D}_{C,P}=\mathrm{coker}(\mathrm{Der}_{\k}(\mathcal{O}_{C,P},\mathcal{O}_{C,P})
\rightarrow\mathrm{Der}_{\k}(\overline{\mathcal{O}}_{C,P},\overline{\mathcal{O}}_{C,P})).
\end{equation}
Here we assume that the ground field has characteristic zero in order that the above map is an inclusion
between the derivations modules, and so we fix\footnote{The number $\theta_P$ is usually detonated by $m_1$, it was introduced
by Deligne in \cite{Del73} when dealing with only one singular point, we changed the notation just to avoid multiples sub-indexes.} $$\theta_{P}:=\dim_{\k}\mathfrak{D}_{C,P} \ \ \mbox{and} \ \ \theta:=\sum_{P\in\mathrm{Sing}(C)}\theta_{P}.$$
Finally, we also consider the \emph{Deligne numbers}: $$e_P:=3\delta_P-\theta_P \ \ \mbox{and} \ \ e:=\sum_{P\in\mathrm{Sing}(C)}e_p.$$

\begin{teo}\label{degnormal}
Let $C\subset\mathbb{P}^r$ be an integral projective curve of degree $d$ whose arithmetic genus is $g$. 
Then the degree of the normal sheaf of $C$ in $\mathbb{P}^r$ is given by
\begin{equation}\label{deg-normal}
\deg(\mathcal{N}_{C/\mathbb{P}^r})=2g-2+(r+1)d+\tau-e. % 2g-2+(r+1)d+\theta+\tau-3\delta
\end{equation}
\end{teo}
\begin{proof}

We start by taking the Euler characteristic in the sequence \eqref{seq-us},
\begin{equation*}\label{eulerchar}
\chi(\mathcal{T}_{C})-\chi(\mathcal{T}_{C/\mathbb{P}^r})+\chi(\mathcal{N}_{C/\mathbb{P}^r})-\chi(\mathrm{T}^1)=0.
\end{equation*}
\noindent The sheaf of differentials $\Omega_{C|\k}$ of $C$ may not be torsion free, but the tangent 
sheaf $\mathcal{T}_{C}$ is a coherent fractional ideal sheaf, so the singular version of the Riemann--Roch Theorem ensures that
$$\chi(\mathcal{T}_{C})=\deg(\mathcal{T}_{C})+1-g.$$ Let $\widetilde{C}$ be the non-singular 
model of $C$ and $\widetilde{g}$ its geometric genus. The 
bundle of differentials $\Omega_{\widetilde{C}|\k}$ is invertible 
of degree $2\widetilde{g}-2$, where $\widetilde{g}=g-\delta$. Thus, $\mathcal{T}_{\widetilde{C}}$ is 
an invertible sheaf of degree $2-2\widetilde{g}$ and so
\begin{equation*}\label{eq1}
\deg(\mathcal{T}_{C})=\deg(\mathcal{T}_{\widetilde{C}})-\theta+\delta=2-2g-e.
\end{equation*}

Since $\mathcal{N}_{C/\mathbb{P}^r}$ and $\mathcal{T}_{C/\mathbb{P}^r}$ have
rank $r-1$ and $r$, respectively, the sequence \eqref{seq-us} and Lemma \ref{lema-deg} imply
that $\deg(\mathcal{T}_{C})-\deg(\mathcal{T}_{C/\mathbb{P}^r})+\deg(\mathcal{N}_{C/\mathbb{P}^r})-\tau=0,$
and so 
%E o resultado segue desde que $\deg(\mathcal{T}_{C})=2-2g-\mu+3\delta$.
\begin{equation*}\label{quasela1}
	\deg(\mathcal{N}_{C/\mathbb{P}^r})-\deg(\mathcal{T}_{C/\mathbb{P}^{r}})= 2g-2+\tau-e.
\end{equation*}
Now we just have to compute de degree of the relative tangent sheaf $\mathcal{T}_{C/\mathbb{P}^r}$, that is given by
$\det(\mathcal{T}_{\mathbb{P}^r}\otimes \mathcal{O}_C)=\det(\mathcal{T}_{\mathbb{P}^r})\otimes
\mathcal{O}_C=\omega_{\mathbb{P}^r}^{\vee}\otimes \mathcal{O}_C= \mathcal{O}_{\mathbb{P}^r}(r+1)\otimes \mathcal{O}_C,$
where $i$ is embedding $i:C\hookrightarrow \mathbb{P}^r$. Hence,
$$\deg(\mathcal{O}_{\mathbb{P}^r}(r+1)\otimes \mathcal{O}_C)=\deg(i^*(\mathcal{O}_{\mathbb{P}^r}(r+1)))=(\deg i)\deg(\mathcal{O}_{\mathbb{P}^r}(r+1))=d(r+1),$$
that concludes the proof.
\end{proof}

\begin{exam}\label{exemplo}
The above formula for the degree of the normal sheaf can be very manageable for suitable classes of curves. For example,
let $\N\subset\mathbb{N}$ be a numerical semigroup of genus $g:=\#(\mathbb{N}\setminus\N)>1$ generated by
$a_1,\dots, a_r$, and $$C_{\N}:=\{(t^{a_1},\dots,t^{a_r})\,;\,t\in\k\}=\Spec \k[\N]\subset\mathbb{A}^{r}$$ be the affine monomial curve associated to $\N$. 
Its is very known that $C_{\N}$ has a unique unibranch singular point $Q=(0,\dots,0)$. Equivalently, the semigroup algebra $\k[\N]$ corresponds uniquely
to an equisingular class of an (analytic germ when $\k=\mathbb{C}$) algebra
of a branch whose singular semigroup is $\N$, see \cite[Corollary 1.2.4 pg. 117]{Z06}. The affine curve $C_{\N}$ is Gorenstein if,
and only if, $\N$ is symmetric, i.e. the largest gap of $\N$ is the biggest possible, $\ell_{g}=2g-1$.

We associated to $C_{\N}$ a projective curve in $\mathbb{P}^r$ by adding just one smooth point at the infinite. Here, we use the same
notation for this projective curve and the affine one. Assuming that $\N$ is symmetric and $2\notin\N$, $C_{\N}$ can be considered as a canonical Gorenstein curve,
i.e. a curve of genus $g$ and degree $2g-2$ in $\mathbb{P}^{g-1}$. So 
all the invariants that appear in formula \eqref{deg-normal} of Theorem \ref{degnormal} are known, with the possible exception of the Tjurina number, namely  $3\delta=3g$, $(r+1)d=g(2g-2)$ and
$\theta=\theta_Q=1+g-\lambda(\N)=g$, where $\lambda(\N):=[\mathrm{End}(\N):\N]=1$ because $\N$ is symmetric, and
by the very definition, $\mathrm{End}(\N)$ is the set of all $n\in\mathbb{N}$ such that
$n+s\in \N\,\forall\,s\in S\setminus\{0\}$.
Hence, $$\deg(\mathcal{N}_{C_{\N}/\mathbb{P}^{g-1}})=g(2g-2)+\tau-2.$$

\noindent Unfortunately, we do not know a general formula for the Tjurina number $\tau$ depending only on the genus of $C_{\N}$, but there is an
implementable method to compute it. We first recall a result due to Herzog \cite{Her70} assuring that the ideal of $C_{\N}$ can be generated 
by isobaric polynomials $F_i$ that are differences of two monomials, namely
$$F_i:=X_{1}^{\alpha_{i1}}\dots X_{r}^{\alpha_{ir}}-X_{1}^{\beta_{i1}}\dots X_{r}^{\beta_{ir}},$$
with $\alpha_i\cdot\beta_i=0$. As usual, the weight of $F_i$ is $d_i:=\sum_j n_j\alpha_{ij}=\sum_j n_j\beta_{ij}$.
For each $i$, let $v_i:=(\alpha_{i1}-\beta_{i1},\dots,\alpha_{ir}-\beta_{ir})$
be a vector in $\k^{r}$ induced by $F_i$. Next a result due to Buchweitz, c.f.  \cite[Thm. 2.2.1]{Bu80}, computes the Tjurina number
for monomial curves.
\begin{teo}[Buchweitz]\label{t1busc}
$\tau=\sum_{s\in\mathbb{Z}}\dim_{\k}T^1_s$, where for each $\l\notin\mathrm{End}(\N)$, 
$$\dim T^{1}_{\l}=\#\{i\in\{1,\dots,r\}\,;\,n_i+\l\notin\N\}-\dim V_{\l}-1,$$
$V_{\l}$ is the subvector space of $\k^{r}$ generated by the vectors $v_i$
such that $d_i+\l\notin\N$. We also have that
$$\dim T^{1}_s=0,\ \forall\,s\in\mathrm{End}(\N).$$
\end{teo}
\noindent Finally, by the very explicit and implementable method in \cite{CF18} and \cite{CS13}, we know that the ideal of $C_{\N}\subseteq\mathbb{P}^{g-1}$
is given by suitable $\frac{1}{2}(g-3)(g-2)$ quadratic and isobaric forms, when the first non zero element $n_1$ of $\N$ is such
that $3<n_1\leq g-1$ and $\N\neq <4,5>$, and by $\frac{1}{2}(g-3)(g-2)$ quadratic and isobaric forms added to
$\binom{g+2}{3}-5g+5$ cubic and isobaric forms in the remaining cases. In this way, one may implement an algorithm to compute the Tjurina
number $\tau$ in all this cases.  To conclude this example, we collect in Table \ref{tabela1} some examples in low genus.
\begin{table}[h]
\begin{center}
\caption{degree of the normal sheaf $\mathcal{N}_{C_{\N}/\mathbb{P}^{g-1}}$ when $\N$ is symmetric and $4\leq g\leq 7$}
\medskip
\begin{tabular}{cccc|cccc}
$\N$ & $g$ & $\tau$ & $\deg(\mathcal{N}_{C_{\N}/\mathbb{P}^{g-1}})$ & $\N$ & $g$ & $\tau$ & $\deg(\mathcal{N}_{C_{\N}/\mathbb{P}^{g-1}})$ \\ \hline
$<3,5>$ & $4$ & $8$ & $30$ & $<6,7,8,9,10>$ & $6$ & $15$ & $73$  \\ \hline
$<4,5,6>$ & $4$ & $8$ & $30$ & $<3,8>$ & $7$ & $14$ & $96$ \\ \hline
$<4,6,7>$ & $5$ & $10$ & $48$  & $<4,7,10>$ & $7$ & $14$ & $96$ \\ \hline
$<5,6,7,8>$ & $5$ & $10$ & $48$  & $<4,6,11>$ & $7$ & $14$ & $96$ \\ \hline
$<3,7>$ & $6$ & $12$ & $70$ & $<5,7,9,11>$ & $7$ & $14$ & $96$ \\ \hline
$<4,6,9>$ & $6$ & $12$ & $70$ & $<5,6,9>$ & $7$ & $14$ & $96$ \\ \hline
$<4,5>$ & $6$ & $12$ & $70$ & $<6,8,9,10,11>$ & $7$ & $17$ & $99$ \\ \hline
$<5,7,8,9>$ & $6$ & $12$ & $70$ & $<7,8,9,10,11,12>$ & $7$ & $21$ & $103$
\end{tabular}
\label{tabela1}
\end{center}
\end{table}

\end{exam}

An immediate consequence of the above global results Theorem \ref{degnormal} and Lemma \ref{dual-lci} is the following known
local result, c.f. \cite[Section 2.6]{Greuel}.

\begin{cor}\label{relacao-invariantes}
If $C\subset\mathbb{P}^r$ is a locally complete intersection integral curve, then $\tau=e$.
\end{cor}

Now let $\mathcal{Z}$ be reduced projective scheme just admitting isolated singularities and assume that  $\eta:\mathcal{Z}\to S$ is a flat morphism whose fibers $\mathcal{Z}_s$ are reduced projective curves. So, under these conditions, we are able
to consider the sheaves $\mathrm{T}^1_{\mathcal{Z}}, \mathcal{D}_{\mathcal{Z}}$ and $\overline{\mathcal{O}}_{\mathcal{Z}}/\mathcal{O}_{\mathcal{Z}}$ as sheaves whose stalks on a fiber $Z_s$ are just $\mathrm{T}^{1}_{Z_s}$, $\mathcal{D}_{Z_s}$ and
$\overline{\mathcal{O}}_{\mathcal{Z}_s}/\mathcal{O}_{\mathcal{Z}_s}$, see \cite[pg.~148]{Z06} for further details. Moreover, these sheaves are flat over $S$. Therefore, we can  apply the upper semicontinuity theorem,
 c.f. \cite[Thm. 12.8]{Har77}, to $s\mapsto \dim H^0(\mathcal{Z}_s,\mathrm{T}^1_{\mathcal{Z}_s})$, $s\mapsto \dim H^0(\mathcal{Z}_s,\mathcal{D}_{\mathcal{Z}_s})$ and
$s\mapsto \dim H^0(\mathcal{Z}_s,\overline{\mathcal{O}}_{\mathcal{Z}_s}/\mathcal{O}_{\mathcal{Z}_s})$, one can easily check the following result.

\begin{prop}\label{semi-tau-mu}
Under the above conditions, the invariants $\tau$, $\theta$ and $\delta$ are upper semicontinuous functions.
\end{prop}

Complete intersections are stable under flat deformations with connected base, provided that the
ground field has characteristic zero, and once the degree of the normal sheaf
of a complete intersection curve only depends on the genus and the degree of the curve, 
c.f. Corollary \ref{relacao-invariantes} and Theorem \ref{degnormal}, it follows that the degree
of the normal sheaf of a complete intersection curve is stable under flat deformations. 

On the other hand, it is easy
to see that the degree of the normal sheaf cannot be an upper semi-continuous function,
just because smooth curves can degenerate to singular ones. But if one keeps the singular
degree under deformations, it is easy to see that the degree of the normal sheaf is a semicontinuos
function as well. 

The notion of a deformation preserving the singularity degree appears in the literature
when we (formally) deform an analytic germ of a singularity. In a survey \cite{Greuel}
by Greuel, the author shows that the singularity degree of an analytic germ is preserved if, and only if,
it is \emph{contractible}, cf. \cite[Corollary 2.46]{Greuel}. We also should mention a result due to Teissier, cf. \cite[2.10 Theorem 3]{Z06},
showing that for a given numerical semigroup $\N$, the positively graded part $\mathrm{T}^{1,+}(\k[\N])$ of the cotangent complex 
is a miniversal space for deformations of the affine monomial curve $C_{\N}$ with reducible base such that each fiber has a singular point whose singular
semigroup is $\N$.

An immediate consequence of Theorem \ref{degnormal} and Proposition \ref{semi-tau-mu} is that, over suitable conditions, the degree of the normal sheaf is an 
upper semicontinuous function.

\begin{cor}\label{semicont}
Let $C\subseteq \mathbb{P}^r$ be a reduced projective curve of degree $d$ and arithmetic 
genus $g$. We have that $\deg(\mathcal{N}_{C/\mathbb{P}^n})$ is an upper semicontinuous function for embedded deformations preserving
the singularity degree $\delta$. 
In other words, if $\eta:\mathcal{X}\rightarrow S$ is a deformation of $C=\mathcal{X}_{\mathbf{0}}$, with $\mathcal{X}\subseteq S\times\mathbb{P}^r$ just admitting isolated singularities, such that each fiber
has constant singularity degree, i.e. $\delta(\mathcal{X}_s)=\delta(C)$ for each $s\in S$, then
$$\deg(\mathcal{N}_{\mathcal{X}_s/\mathbb{P}^r})\leq \deg(\mathcal{N}_{C/\mathbb{P}^r})=2g-2+(r+1)d+\tau-e.$$
%For abuse of notation we are considering $\mathcal{C}_s$ 
%as a curve in $\mathbb{P}^r:=\mathbb{P}^r_{k(s)}.$ 
\end{cor}

\bibliographystyle{amsalpha}

\begin{thebibliography}{10}

	
%\bibitem[Akh]{Akh} Mathew, Akhil. \textit{Flatness, semicontinuity, and base-change}, (2011).
\bibitem[AFO16]{AFO16} M. Aprodu, G. Farkas, and A. Ortega. Restricted Lazarsfeld-Mukai bundles
and canonical curves,
\textit{Development of moduli theory in: Advanced
studies in pure mathematics. Mathematical Society of Japan}, 96 (2016) pp. 303--322.


\bibitem[AC81]{AC81} E. Arbarello and M. Cornalba, Footnotes to a paper of Beniamino Segre, \textit{Math. Ann.}, 256, (1981), 341--362.

\bibitem[ACGH85]{ACGH85} E.\ Arbarello, M.\ Cornalba, P.A.\ Griffiths and J.\ Harris, Geometry of algebraic 
curves, Vol I, {\em Grund\-lehren der Mathematischen Wissenschaften} 267, Springer-Verlag, New York (1985).

%\bibitem[ACG11]{ACG11} E.\ Arbarello, M.\ Cornalba, and P.A.\ Griffiths, Geometry of algebraic curves, \textit{Grund\-lehren der Mathematischen Wissenschaften} 268, Vol $\mathrm{II}$, Springer-Verlag, New York, (2011).

%\bibitem[Art76]{Art76} M. Artin, Lectures on Deformation of Singularities, \textit{Tata Institute of Fundamental Research Bombay}, (1976).

%\bibitem[Bas77]{Bas77} R. Bassein, On smoothable curve singularities: local methods, \textit{Math. Ann 230} (1977),
%273--277.

%\bibitem[Ber98-1]{Ber98-1} R. Berger, Behavior of the torsion of the differential module of an algebroid curve under 
%quadratic transformations. \textit{arXiv:math/9805044v1} (1998).

%\bibitem[Ber98-2]{Ber98-2} R. Berger, Report on the torsion of the differential module of an algebraic curve. \textit{arXiv:alg-geom/9205001v2} (1998).

\bibitem[Bra97]{Bra97} J. Brawner, Tetragonal Curves, Scrolls and $K3$ Surfaces. \textit{Trans. Am. Math. Soc.}, 349 (1997) 3075-3091.

\bibitem[Bru17]{Bru17} G. Bruns, The normal bundle of canonical genus 8 curves. 
\textit{arXiv preprint:} 
\url{https://arxiv.org/abs/1703.06213} (2017).

\bibitem[Buc80]{Bu80} {R-O.\ Buchweitz},
On deformations of monomial curves,
{\em Lecture notes in Mathematics }777 (1980) 205--220.

\bibitem[CS13]{CS13} A. Contiero and KO. St\"ohr, 
Upper bounds for the dimension of moduli spaces of curves with symmetric Weierstrass semigroups,
\textit{J. London Math. Soc.} 88 (2013) 580--598.

\bibitem[CF18]{CF18}A.\ Contiero and A. Fontes, On the locus of curves with 
an odd subcanonical marked point, \textit{arXiv preprint:} \url{https://arxiv.org/abs/1804.09797} (2018).

%\bibitem[CFSV18]{CFSV18}A.\ Contiero, A. Fontes, J. Stevens and J. Q. Vargas, 
%On the dimension of the stratum of the moduli of pointed curves by Weierstrass gaps, \textit{arXiv preprint}, (2018).

%\bibitem[Cop92]{Cop92} M. Coppens, Free Linear systems on Integral Gorenstein Curves, \textit{J. Algebra} 145 (1992), 209-218.


\bibitem[Del73]{Del73} P. Deligne,
Intersections sur les surfaces r\'eguli\`eres \textit{SGA 7, Expos\'e X, Lectures notes in Mathematics} 340 (1973) 1--37.

%\bibitem[DK73]{Del73} P. Deligne and N. Katz, Groupes de Monodromie en G\'eom\'etrie Alg\'ebrique. 
%SGA 7 II. Lecture Notes in Mathematics 340. Berlin, Heidelberg, New York: Springer, (1973).
%Intersections sur les surfaces r\'eguli\`eres \textit{SGA 7, Expos\'e X, Lectures notes in Mathematics}, 340, (1973), 1--37.

%\bibitem[EH]{EH} D. Eisenbud and J. Harris, On varieties of minimal degree. \textit{Proc. Symp. Pure Math.} 46, (1987), 3--13. 

%\bibitem[For11]{For11} F. Forstnerič, Stein manifolds and holomorphic mappings. \textit{Springer Berlin Heidelberg}, (2011).

\bibitem[Fer01]{Rita0} R. Ferraro, Weil divisors on rational normal scrolls,
\emph{Geometric and combinatorial aspects of Commutative Algebra, Lectures Notes in Applied and Pure Mathematics}, Vol.
217, Marcel Dekker Ser., New York (2001) 183--198.

\bibitem[Ful69]{F69}W. Fulton, Hurwitz schemes and irreducibility of moduli of algebraic curves,
\emph{Ann. Math.} 90 (1969) 542--575.

%\bibitem[Fer04]{Rita}R. Ferraro,
%Linkage on arithmetically Cohen–Macaulay schemes with application to the classication of curves of maximal genus,
%\emph{J. P. App. Algebra}, 188 (2004) 95--115.

%\bibitem[FL85]{FL85} W. Fulton and S. Lang. Riemann-Roch Algebra, \textit{Springer-Verlag}, New York, (1985).

\bibitem[Gre20]{Greuel} GM Greuel,
Deformation and Smoothing of Singularities. 
\emph{Handbook of Geometry and Topology of Singularities I}, 
Springer Verlage, Cham. (2020) 389--448.

\bibitem[Gro59]{Gro59} A. Grothendieck, 
Th\'eor\`emes de dualit\'e pour les faisceaux algebraic coh\'erents. \textit{Bourbaki}, 1959.

\bibitem[Har77]{Har77} R. Hartshorne, Algebraic geometry. 
\emph{Graduate Texts in Mathematics}, Springer Verlage, 1977.

\bibitem[Her70]{Her70} J. Herzog, Generators and relations of abelian semigroups and semigroup rings. 
\emph{Manuscripta Math} 3 (1970) 175--193. 

\bibitem[KM09]{KM09} S. Kleiman and R.V. Martins, The canonical model of a singular curve. \textit{Geometriae Dedicata}, 139 (2009) 139--166.


\bibitem[LMS19]{LMS19} D. Lara, Martins, R. and  J. Souza, M. On gonality, scrolls, and canonical models of 
non-Gorenstein curves. \textit{Geometriae Dedicata}, (2019) 1--23.

\bibitem[LS67]{LS67} S. Lichtenbaum and M. Schlessinger, 
The Cotangent Complex of a Morphism, 
\textit{Trans. Am. Math. Soc.} 128 (1967) 41--70.

\bibitem[Muk95]{Muk95} S. Mukai. Curves and symmetric spaces, I. \textit{American Journal of Mathematics}, 117 (1995) 1627--1644.

\bibitem[MI03]{MukIde} S. Mukai e M. Ide,
Canonical curves of genus eight, 
\emph{Proceedings of the Japan Academy} 79 (2003) 59-64.

\bibitem[Muk10]{Muk10} S. Mukai. Curves and symmetric spaces, II. \textit{Ann. of Math.} 172 (2010)
%1539--1558.

%\bibitem[PeVa]{PV} N. Penev e R. Vakil,
%The Chow ring of the moduli space of curves of genus 6,
%\emph{Alg. Geometry} (2015) 123--136.

%\bibitem[Pfl18]{Pfl18} N.\ Pflueger, On non-primitive Weierstrass points,
%{\em  Algebra and Number Theory}, 12 (2018) 1923--1947.

\bibitem[Pin74]{Pin74} H.\ Pinkham, Deformations of algebraic varieties with $G\sb{m}$-action', {\em Ast\'erisque 20}, (1974).

%\bibitem{NP2}{N.\ Pflueger},`',{\em  J. Alg.}, ?? (2019) ??.

%\bibitem[Rim72]{Rim72} D. S. Rim, Torsion differentials and deformation, {\em Trans. Amer. Math. Soc.} 169, (1972,) 257--278. 

%\bibitem[RV77]{RV77} D.S. Rim and  M.A. Vitulli,
%Weierstrass points and monomial curves,
%{\em J.\ Algebra 48 }(1977) 454--476.

%\bibitem[Schl68]{Schl68} M. Schlessinger, Functors of Artin rings, \textit{Trans. Amer. Math. Soc.} 130, (1968), 208--222.
%
\bibitem[Sch86]{Schr86} F. O. Schreyer.
Syzygies of canonical curves and special linear series. 
\emph{Math. Ann.}, 275 (1986) 105--137.
%
%\bibitem[Schr91]{Schr91} F. O. Schreyer. 
%A standard basis approach to syzygies of canonical curves, 
%\emph{J. fur Reine. Ang. Math.} 421 (1991) 830--123.

%\bibitem[Sto93-1]{Sto93-1} K. O. St{\"o}hr, On the Poles of Regular Differentials of Singular Curves, \textit{Bol. Soc. Bras. Mat}, 24, (1993), 105--136.

\bibitem[Sto93]{Sto93-2} K. O.\ St{\"o}hr,
On the moduli spaces of Gorenstein curves with symmetric Weierstrass semigroups,
{\em J.\ reine angew.\ Math.\ }441 (1993) 189--213.

\bibitem[Tru18]{Tru18} C. C. Trung, Gorenstein Rings, \textit{Bachelor Thesis}, Ho Chi Minh City, 66 (2018)

\bibitem[Vie77]{Vie77} E. Viehweg, Rational singularities of higher dimensional schemes,
\textit{Proc. Amer. Math. Soc.} 63 (1977) 6--8.

%\bibitem[V19]{J19} J. Vargas, On Deligne-Pinkham's bound. \textit{PhD Thesis}, UFMG, Belo Horizonte, (2019).

\bibitem[Za06]{Z06} O. Zariski, with an appendix by B. Teissier,
The moduli problem for plane branches,
\textit{American Mathematical Society}, University Lecture Series, 39 (2006) 111--151.


%\bibitem{EH87} {D.\ Eisenbud and J.\ Harris},
% `Existence, decomposition, and limits of certain Weierstrass points',
% {\em Invent.\ Math.\ }87 (1987) 495--515.
% \bibitem{NP2}{N.\ Pflueger},`???',
% {\em arXiv preprint}, arXiv:??? (2016).
%\bibitem{Rim}{D. Rim}
% `Formal deformation theory' (SGA 7,Expos\'e VII),
% {\em Lectures notes in Mathematics }288(1972) 32--132.

\end{thebibliography}

\bigskip

\parbox[t]{3in}{{\rm Andr\'e Contiero}\\
{\tt \href{mailto:contiero@ufmg.br}{contiero@ufmg.br}}\\
{\it Universidade Federal de Minas Gerais}\\
{\it Belo Horizonte, MG, Brazil}} \hspace{0.9cm}
\parbox[t]{3in}{{\rm Aislan L. Fontes}\\
{\tt \href{mailto:aislan@ufs.br }{aislan@ufs.br }}\\
{\it Universidade Federal de Sergipe}\\
{\it Itabaiana, SE, Brazil}} 

\bigskip

\parbox[t]{3in}{{\rm J\'unio Teles}\\
{\tt \href{mailto:telesjunio@gmail.com }{telesjunio@gmail.com }}\\
{\it Universidade Federal de Minas Gerais}\\
{\it Belo Horizonte, MG, Brazil}} \hspace{0.9cm}

\end{document}